\documentclass[10pt]{amsart}

\usepackage{latexsym}
\usepackage{amsthm,amssymb,amsmath}
\usepackage{tikz}
\usepackage[margin=2.5cm]{geometry}

\usepackage[colorlinks=true,linkcolor=blue,citecolor=blue,filecolor=blue,menucolor=blue,urlcolor=blue]{hyperref}

\renewcommand{\to}{\rightarrow}

\newcommand{\R}{\mathbb{R}}
\newcommand{\N}{\mathbb{N}}
\newcommand{\Z}{\mathbb{Z}}
\newcommand{\C}{\mathcal{C}}
\newcommand{\attach}{\lhd}

\theoremstyle{plain}
\newtheorem{thm}{Theorem}[section]
\newtheorem{cor}[thm]{Corollary}
\newtheorem{pro}[thm]{Problem}

\newtheorem{lem}[thm]{Lemma}
\newtheorem{prop}[thm]{Proposition}

\theoremstyle{definition}
\newtheorem{defn}[thm]{Definition}

\begin{document}

\begin{abstract}
Whiteley \cite{wh} gives a complete characterization of the infinitesimal flexes of complete bipartite frameworks. Our work generalizes a specific infinitesimal flex to include joined graphs, a family of graphs that contain the complete bipartite graphs. We use this characterization to identify new families of counterexamples, including infinite families, in $\R^5$ and above to Hendrickson's conjecture on generic global rigidity.
\end{abstract}

\title{Rigidity of Graph Joins and Hendrickson's Conjecture}
\thanks{Thanks to Dylan Thurston, Andrew Fanoe, and Kiril Ratmanski for their support and suggestions throughout this project. This paper was partially supported by NSF RTG Grant DMS 07-39392.}
\author{Timothy Sun}
\address{Department of Mathematics, Columbia University, New York, NY 10027}
\email{ts2578@columbia.edu}
\author{Chun Ye}
\address{Department of Mathematics, Columbia University, New York, NY 10027}
\email{cy2214@columbia.edu}
\maketitle

\section{Introduction}

A $d$-dimensional framework of a graph is a mapping from the vertices of the graph to points in Euclidean $d$-space. A natural question to ask is whether a graph is \emph{locally rigid}, i.e. can we can move the vertices of the framework while preserving edge lengths? Furthermore, when a framework is locally rigid, another question to ask is whether the graph is \emph{globally rigid}, i.e. do the edge lengths uniquely define a framework up to Euclidean motions? 

Hendrickson \cite{bh} found two necessary conditions for a graph to be generically globally rigid and conjectured that they were also sufficient. Connelly \cite{cn} discovered a family of complete bipartite graphs in $\R^3$ and higher that were counterexamples to Hendrickson's conjecture:

\begin{thm}[Connelly \cite{cn}]
\label{cg}
If $a,b\geq d+2$ and $a+b=\binom{d+2}{2}$, then $K_{a,b}$ is generically almost-globally rigid in $\R^d$.
\end{thm}

Work has been done on identifying counterexamples that are subgraphs of this family in \cite{fj}. Our work extends Theorem \ref{cg} in the opposite direction, exhibiting a family of counterexamples that have Connelly's graphs as subgraphs. Connelly and Whiteley \cite{cw} showed that a graph operation known as coning preserves local and global rigidity. In particular, coning can be used to construct new counterexamples in higher dimensions from known counterexamples. We identify counterexamples that are subgraphs of coned graphs.

Frank and Jiang \cite{fj} found a graph that could be ``attached" to graphs that are sufficiently rigid in $\R^5$ to form an infinite number of counterexamples to Hendrickson's conjecture. However, one step of the proof was aided by a computer program, so their result could not be immediately generalized to higher dimensions. We give a conceptual proof of generic local rigidity for their graph and similar graphs in order to exhibit graph attachments in higher dimensions.

In this paper, we introduce the notion of the \emph{quadric rigidity matrix}, which generalizes one of Whiteley's \cite{wh} conditions for infinitesimal rigidity. We use the quadric rigidity matrix to characterize all infinitesimal flexes of balanced joined graphs and for the construction of the aforementioned families of graphs. 

\section{Graph Theory Preliminaries}

A \emph{graph} $G = (V,E)$ is a 2-tuple consisting of a set $V = \{v_1, v_2,..., v_{|V|}\}$ of vertices and a set $E \subseteq V^{(2)}$ of edges between the vertices. From our choice of how we defined the edge set, all graphs in this paper are undirected and simple. We denote an edge connecting vertices $v_i$ and $v_j$ as $v_iv_j$ and say that $v_i$ and $v_j$ are \emph{adjacent}.

If a subgraph has the same vertex set, we call such a subgraph a \emph{factor}. The \emph{edge complement} of a graph $G=(V,E)$, denoted $\overline{G} = (V,E')$, is the graph where $v_iv_j \in E' \Leftrightarrow v_iv_j \not\in E$, or equivalently, $E' = V^{(2)}-E$. Two graphs $G$ and $H$ are \emph{isomorphic} if there exists a bijective function $\phi : V_G \to V_H$ such that $v_1v_2 \in E_G$ if and only if $\phi(v_1)\phi(v_2) \in E_H$. 

A graph is connected if, for all pairs of vertices $v_i$ and $v_j$, there exists a path of vertices starting from $v_i$ and ending at $v_j$. A graph is \emph{$k$-(vertex)-connected} if deleting any subset of $k-1$ vertices and edges incident on those vertices results in a connected graph.

The \emph{disjoint union} of two graphs $G$ and $H$, denoted $G\cup H$, is the graph formed by the disjoint union of the vertex sets and edge sets. The \emph{graph join} of graphs $G$ and $H$, denoted $G+H$, is the graph whose vertex set is $V_G \cup V_H$ and whose edge set is $E_G \cup E_H \cup \{v_gv_h | v_g \in V_G, v_h \in V_H\}$. That is, $G+H$ results from taking $G$ and $H$ and adding all possible edges between vertices of $G$ and vertices of $H$. We call such a graph a \emph{joined graph}, and any edge in $E_G \cup E_H$ is \emph{extraneous}. We will call a joined graph $G+H$ \emph{balanced} if $|V_G|, |V_H| \geq d+1$.

\begin{figure}[ht]
\begin{tikzpicture}
[v/.style={circle,draw=black!100,fill=red!100,thick,inner sep=2pt}]
\foreach \i in {1,...,6} \node (v\i) at (2*{cos(60*\i)},2*{sin(60*\i)})[v]{};
\foreach \i in {1,...,6} \foreach \j in {1,...,6} \draw [-] (v\i) to (v\j);
\end{tikzpicture}
\caption{A complete graph.}
\label{fig-complete}
\end{figure}
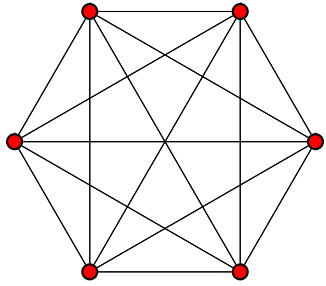

The graph join operation allows us to define familiar terms in new ways. The \emph{complete graph} on $i$ vertices, denoted $K_i$, is defined recursively, where $K_1$ is a single vertex, and $K_i = K_1 + K_{i-1}$ for $i > 1$. Figure \ref{fig-complete} is the graph $K_6$. For convenience, we define $E_i := \overline{K_i}$, the graph on $i$ vertices with no edges. The \emph{complete bipartite graph} on $a$ and $b$ vertices, denoted $K_{a,b}$, is $E_a + E_b$. $V_{E_a}$ and $V_{E_b}$ are referred to as the two \emph{bipartite classes}. Figure \ref{fig-bipartite} contains examples of a complete bipartite graph and a joined graph.

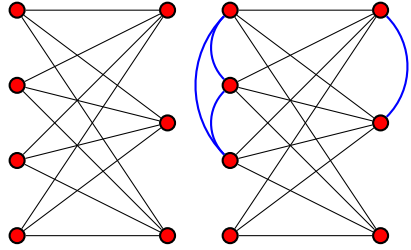
\begin{figure}[ht]
\begin{tikzpicture}
[v/.style={circle,draw=black!100,fill=red!100,thick,inner sep=2pt}]
\foreach \i in {1,...,4} \node (v\i) at (0,3-\i)[v]{};
\foreach \i in {5,...,7} \node (v\i) at (2,2-{(\i-5)*1.5})[v]{};
\foreach \i in {1,...,4} \foreach \j in {5,...,7} \draw [-] (v\i) to (v\j);
\end{tikzpicture}
\begin{tikzpicture}
[v/.style={circle,draw=black!100,fill=red!100,thick,inner sep=2pt}]
\foreach \i in {1,...,4} \node (v\i) at (0,3-\i)[v]{};
\foreach \i in {5,...,7} \node (v\i) at (2,2-{(\i-5)*1.5})[v]{};
\foreach \i in {1,...,4} \foreach \j in {5,...,7} \draw [-] (v\i) to (v\j);
\draw[blue, thick, bend right=45] (v1) to (v2);
\draw[blue, thick, bend right=45] (v1) to (v3);
\draw[blue, thick, bend right=45] (v2) to (v3);
\draw[blue, thick, bend left=45] (v5) to (v6);
\end{tikzpicture}
\caption{The complete bipartite graph $K_{4,3}$ and the joined graph $(K_3\cup K_1)+(K_2\cup K_1)$, respectively. }
\label{fig-bipartite}
\end{figure}

A \emph{vertex amalgamation} $(G;u_1,u_2,...,u_i)*(H;v_1,v_2,...,v_i)$ is the graph $(G\cup H)/R$, where $R$ is the equivalence relation $\{u_1=v_1, u_2=v_2,...,u_i=v_i\}$. Intuitively, a vertex amalgamation takes vertices of two graphs and pastes them together to get the resulting graph, as in Figure \ref{fig-amal}.

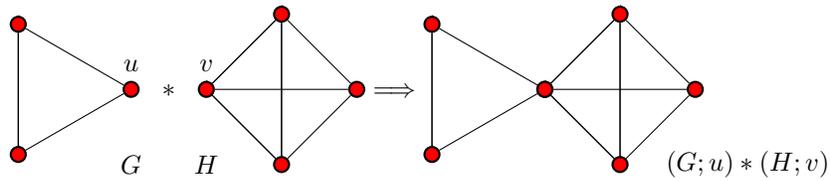
\begin{figure}[ht]
\begin{tikzpicture}
[v/.style={circle,draw=black!100,fill=red!100,thick,inner sep=2pt}]
\foreach \i in {1,2,3}{
  \node (v\i) at ({cos(120*(\i -1))},{sin(120*(\i -1))})[v]{};
  \node (w\i) at (5.5+{cos(120*(\i -1))},{sin(120*(\i -1))})[v]{};
}
\foreach \i in {4,...,7}{
  \node (v\i) at (3+{cos(90*\i)},{sin(90*\i)})[v]{};
  \node (w\i) at (7.5+{cos(90*\i)},{sin(90*\i)})[v]{};
}
\node at (1.5,-0.06){*};
\node at (4.5,-0.04){$\Longrightarrow$};
\node [above] at (v1.north){$u$};
\node [above] at (v6.north){$v$};
\node at (1,-1){$G$};
\node at (2,-1){$H$};
\node at (9.2,-1){$(G;u)*(H;v)$};
\foreach \i in {1,2,3}
  \foreach \j in {2,3}{
    \draw [-] (v\i) to (v\j);
    \draw [-] (w\i) to (w\j);
  }
\foreach \i in {4,...,7}
  \foreach \j in {5,...,7}{
    \draw [-] (v\i) to (v\j);
    \draw [-] (w\i) to (w\j);
  }
\end{tikzpicture}
\caption{The vertex amalgamation $(G;u)*(H;v)$.}
\label{fig-amal}
\end{figure}

\section{Frameworks}

A \emph{$d$-dimensional framework} is a 2-tuple $(G,p)$ where $G$ is a graph and $p$ is a mapping, known as a \emph{configuration}, that takes elements of $V_G$ to $\R^d$. We assume that for a configuration, not all the vertices lie on a hyperplane. Two frameworks $(G,p) = (G; p_1, p_2,...,p_v)$ and $(G,q) = (G;  q_1, q_2,...,q_v)$ are \emph{equivalent} if for all pairs of adjacent vertices $v_i$ and $v_j$, $||p_i-p_j||=||q_i-q_j||$. They are \emph{congruent} if all pairwise distances between points are equal. A \emph{generic configuration} is a mapping in which the coordinates of the vertices are algebraically independent over $\Z$; that is, no non-trivial polynomial with integer coeffients over the coordinates is 0. A \emph{generic framework} is a framework whose configuration is generic.

A framework $(G,p)$ is said to be \emph{globally rigid} if any equivalent framework $(G,q)$ is also congruent. Alternatively, any equivalent configuration can be reached by some Euclidean motion. A framework $(G,p)$ is said to be \emph{locally flexible} if there exists a parametric curve in $\R^{vd}$ of equivalent configurations that is not a Euclidean motion. A graph that is not locally flexible is \emph{locally rigid}.

A graph is \emph{generically locally rigid} (GLR) if any generic framework is locally rigid. Similarly, a graph is \emph{generically globally rigid} (GGR) if any generic framework is globally rigid. A graph is \emph{generically redundantly rigid} (GRR) if deleting any edge from the graph leaves a GLR graph. 

\begin{prop}
\label{ggrsub}
If $G=(V,E)$ is a graph that is not generically globally rigid, then any factor $G'$ is also not generically globally rigid.
\end{prop}
\begin{proof}
Suppose there existed two equivalent, non-congruent frameworks $(G,p)$ and $(G,q)$ for generic $p$ and $q$. Then $(G',p)$ and $(G', q)$ are equivalent, non-congruent frameworks.
\end{proof}

The following theorems demonstrate that generic local and global rigidity are properties of the underlying graph, and not the framework.

\begin{thm}
If any generic framework of a graph $G$ is locally rigid, then all generic frameworks of $G$ are locally rigid.
\end{thm}

The above result is a corollary of Theorem \ref{inf}.

\begin{thm}[Connelly \cite{c2}, Gortler, Healy, Thurston \cite{ght}]
If any generic framework of a graph $G$ is globally rigid, then all generic frameworks of $G$ are globally rigid.
\end{thm}

For non-generic frameworks, there are problems like all points lying on a hyperplane that might yield unexpected rigidity properties. Thus, we consider only generic configurations because we can give characterizations of rigidity based on the underlying graph alone. An example of such a characterization comes from Hendrickson \cite{bh}, who found necessary conditions for a graph to be generically globally rigid and conjectured that they were also sufficient.

\begin{thm}[Hendrickson \cite{bh}]
\label{hend}
If a non-complete graph $G$ is generically globally rigid in $\R^d$, then it is $(d+1)$-connected and generically redundantly rigid.
\end{thm}

A graph that is GGR requires $(d+1)$-connectivity because if the graph could be disconnected into two components by deleting $d$ vertices, reflecting one component across the hyperplane defined by those $d$ points yields an equivalent, but not congruent framework. A graph that is GGR requires redundant rigidity because otherwise, we can delete some non-redundant edge, flex the graph, and replace the edge with the same length to get a non-congruent framework. This is impossible for some frameworks, but Hendrickson demonstrates that they are not generic since they lie on critical points of a manifold.

Connelly \cite{cn} and Frank and Jiang \cite{fj} found families of counterexamples to Hendrickson's conjecture. Such a counterexample is said to be \emph{generically almost-globally rigid}\footnote{Frank and Jiang \cite{fj} refer to these graphs as \emph{generically partially rigid}.}. We will generalize these results in the remainder of this paper.

\section{Infinitesimal Flexes and Equilibrium Stresses}
Let $f_G: R^{vd} \to R^{e}$ be a mapping where we take the coordinates of the configuration and output the edge-length squared of each edge. That is, $f_{G}(p=(p_1,p_2,...,p_v)) = (...,||p_i-p_j||^2,...)$. The \emph{rigidity matrix} of a framework is the Jacobian $df_G(p)$ and has dimensions $e \times vd$. For example, the rigidity matrix for the graph $K_3$ with coordinates $p_1 = (0,2), p_2 = (2, -2), p_3 = (1, 3)$ could be written as

$$2*\bordermatrix{         & p_{1,x} & p_{1,y} & p_{2,x} & p_{2,y} & p_{3,x} & p_{3,y}  \cr
                    v_1v_2 & -2      & 4       & 2       & -4      & 0       & 0        \cr
                    v_1v_3 & -1      & -1      & 0       & 0       & 1       & 1        \cr
                    v_2v_3 & 0       & 0       & 1       & -5      & -1      & 5        }.$$
                    
Since all possible $f_G$ are permutations of each other, the rigidity matrix is unique up to row permutations. An \emph{infinitesimal motion} is an element of the kernel of $df_G(p)$. Equivalently, an infinitesimal motion $p' = (p_1', p_2',..., p_v')$ satisfies, for any edge $v_iv_j$, $(p_i - p_j) \cdot (p_i' - p_j') = 0$. Infinitesimal motions generalize the notion of Euclidean motions and local flexes. To see this, consider the time derivative of $f_G$. Since Euclidean motions and local flexes preserve edge lengths, we wish to have, for any edge $v_iv_j$,

\begin{align*}
\frac{d}{dt} (p_i-p_j)^2 &= \frac{d}{dt} (p_i-p_j)\cdot (p_i-p_j)\\
                         &= 2(p_i - p_j) \cdot \frac{d}{dt} (p_i-p_j) \\
                         &= 2[(p_i-p_j) \cdot (p_i'-p_j')],
\end{align*}

which is zero if it is an infinitesimal motion. Any infinitesimal motion that is not a Euclidean motion is an \emph{infinitesimal flex}. A graph with no infinitesimal flex is \emph{infinitesimally rigid}.  Using the equivalence of the two definitions of an infinitesimal flex, Asimow and Roth \cite{ar} proved the following theorems that demonstrate the connection between the local rigidity and the rigidity matrix.

\begin{thm}[Asimow and Roth \cite{ar}]
A framework $(G,p)$ is infinitesimally rigid if and only if the rank of its rigidity matrix is $vd-\binom{d+1}{2}$, or equivalently, if the nullity is $\binom{d+1}{2}$. 
\end{thm}

Since the Euclidean motions are infinitesimal motions and have dimension $\binom{d+1}{2}$, $\binom{d+1}{2}$ is the smallest possible dimension for the kernel, which is the best possible. 

\begin{thm}[Asimow and Roth \cite{ar}]
\label{inf}
A graph with at least $d+1$ vertices is generically locally rigid if and only if a generic framework of it is infinitesimally rigid.
\end{thm}

\begin{prop}
\label{addedge}
If $G=(V,E)$ is a graph that is generically locally rigid, then adding an edge yields a generically locally rigid graph.
\end{prop}
\begin{proof}
Adding a row to the rigidity matrix cannot decrease the rank, but since the rank is already $vd-\binom{d+1}{2}$, the resulting graph's rank is also $vd-\binom{d+1}{2}$. 
\end{proof}

\begin{prop}
\label{lat}
Given a graph $G$ that is generically locally rigid in $\R^d$, adding a vertex $v$ to $G$ and at least $d$ edges connected to that vertex yields a generically locally rigid graph.
\end{prop}
\begin{proof}
We only need to consider the case where we add $d$ edges, since adding more follows from Proposition \ref{addedge}. Consider the rigidity matrix of $G$. Adding $v$ increases both the column and row size by $d$. The resulting matrix is block triangular, so consider the submatrix formed by the newly added rows and columns. The determinant of the submatrix forms an algebraic equation in the coordinates and hence must be non-zero for a generic placement of the new vertex. Thus, the submatrix is of maximal rank and the resulting graph is also GLR. 
\end{proof}

An \emph{(equilibrium) stress} is a vector $\omega = (...,\omega_{ij},...) \in \R^e$ such that for all vertices $v_i \in V$,

\[\displaystyle\sum_{j\mid v_iv_j\in E} \omega_{ij}(p_i-p_j) = 0.\]

By multiplying out $(df_G)^T \omega$, we find that this definition is equivalent to saying that $\omega \in \ker(df_G)^T$. We denote the space of stresses as $\Omega(G,p)$. From these definitions, $\dim \ker df_G(p) = vd - e + \dim \Omega(G,p)$ by a matrix dimension argument. This yields a crucial characterization of redundant edges.

\begin{prop}[Frank and Jiang \cite{fj}]
\label{red}
Removing an edge $e$ of a generically locally rigid graph $G$ preserves local rigidity if and only if for any $r \in \R$, there exists a stress with value $r$ on $e$.
\end{prop}
\begin{proof}
Assume both $G$ and $G-\{e\}$ are GLR. Then the space of flexes for both graphs is $\binom{d+1}{2}$, so adding edge $e$ to $G-\{e\}$ increases $\dim \Omega(G,p)$. Thus, a stress with non-zero value on $e$ must exist. Conversely, deleting an edge with non-zero stress decreases the space of stresses by at least 1 because scaling that stress creates a one-dimensional subspace, so $\binom{d+1}{2} = \dim \ker df_G(p) \geq \dim \ker df_{G-\{e\}}(p)$. Because the Euclidean motions are infinitesimal flexes of all frameworks and have dimension $\binom{d+1}{2}$, we have equality.
\end{proof}

If we restrict ourselves to only balanced complete bipartite graphs, we obtain a tidy characterization of the stresses and flexes. The following theorem is a formula for the dimension of the stresses.

\begin{thm}[Bolker and Roth \cite{br}]
\label{brz}
Given some balanced complete bipartite graph $K_{a,b}$ where $a+b \leq \binom{d+2}{2}$, the space of stresses $\dim \Omega(K_{a,b})$ for a generic configuration has dimension $(a-d-1)(b-d-1)$.
\end{thm}

\begin{cor}
\label{bir}
If $K_{a,b}$ is generically locally rigid and $a,b \geq d+2$, then $K_{a,b}$ is generically redundantly rigid.
\end{cor}
\begin{proof}
Since $\dim \Omega(K_{a,b}) > 0$, there must exist a stress which is non-zero on some edge. That edge is then redundant by Proposition \ref{red}, so by symmetry, all the edges of $K_{a,b}$ are redundant.
\end{proof}
Whiteley \cite{wh} explicitly describes the infinitesimal flexes that arise from the stresses of a complete bipartite framework. When $v < \binom{d+2}{2}$, there exists at least one \emph{quadric surface} that passes through all $v$ points. A quadric surface is a $(d-1)$-dimensional surface in $\R^d$ whose space is the locus of zeroes of some quadratic polynomial in $d$ variables. That is, a quadric surface can be viewed as the set of all points $p=(p_1,p_2,...,p_d)$ that satisfy the equation

\[\displaystyle\sum^d_{i=1}A_ip_i^2 + \displaystyle\sum_{j=1}^d\sum_{k=j+1}^d 2B_{j,k}(p_jp_k) + \displaystyle\sum^d_{l=1}C_lp_l + D = 0\]

for some real coefficients $A_i$, $B_{j,k}$, $C_l$, $D$ not all zero. A quadric surface can also be defined as the set $\{p \in \R^d \mid (p,1)^TQ(p,1) = 0\}$ for some symmetric $(d+1)\times(d+1)$ matrix. To see that this definition is equivalent to the polynomial equation counterpart, we let $A_i := [Q]_{i,i}$, $B_{j,k} := 2[Q]_{j,k}$, $C_l := 2[Q]_{d+1,l}$, $D := [Q]_{d+1,d+1}$. Expanding out $(p,1)^TQ(p,1)$ yields the polynomial equation definition.

\begin{figure}[ht]
\begin{tikzpicture}
[v/.style={circle,draw=black!100,fill=black!100,thick,inner sep=2pt},
 w/.style={circle,draw=black!100,fill=white!100,thick,inner sep=2pt}
]
\draw[draw=black,dashed,] (0,0) circle (1);
\foreach \i in {1,3,5}{
  \node (v\i) at ({cos(40*\i+2*\i*\i)},{sin(40*\i+2*\i*\i)})[v]{};
  \draw [->,blue,thick] (v\i) to ({cos(40*\i+2*\i*\i)*0.5},{sin(40*\i+2*\i*\i)*0.5});
}
\foreach \i in {2,4,6}{
  \node (v\i) at ({cos(40*\i+2*\i*\i)},{sin(40*\i+2*\i*\i)})[w]{};
  \draw [->,blue,thick] (v\i) to ({cos(40*\i+2*\i*\i)*1.5},{sin(40*\i+2*\i*\i)*1.5});
}
\foreach \i in {1,3,5}
  \foreach \j in {2,4,6}
    \draw [-] (v\i) to (v\j);
\end{tikzpicture}
\caption{$K_{3,3}$ on a circle and its corresponding quadric flex. The bipartite classes are shown in different colors. }
\label{fig-white}
\end{figure}
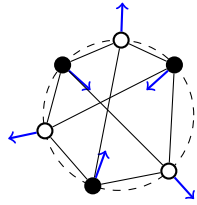

Let the \emph{quadric flex} be the flex $Qx_i$ for all vertices $x_i$ of one bipartite class, and $-Qy_i$ for all vertices $y_i$ of the other bipartite class. Intuitively, this flex pushes one bipartite class into the quadric surface and the other class outwards from the surface, as seen in Figure \ref{fig-white}. To see that this is in fact an infinitesimal flex, for any pair of adjacent vertices $v_a$, $v_b$,

\begin{align*}
(p_a-p_b)\cdot (Qp_a-(-Qp_b)) &= p_a\cdot Qp_a + p_b\cdot Qp_b + p_a\cdot Qp_b - p_b\cdot Qp_a \\
                              &= p_a\cdot Qp_b - p_b^TQp_a,
\end{align*}

but since $Q$ is symmetric, $p_a\cdot Qp_b = p_a^TQp_b = p_b^TQp_a = p_b\cdot Qp_a$, so $(p_a-p_b)\cdot (Qp_a-(-Qp_b)) = 0$.

\begin{prop}[Whiteley \cite{wh}]
\label{inout}
For generic $d$-dimensional frameworks of balanced complete bipartite graphs with fewer than $\binom{d+2}{2}$ vertices, the quadric flexes spans the space of all infinitesimal flexes modulo Euclidean motions.
\end{prop}

To see why the quadric flexes are the only such flexes, consider the space of flexes for bipartite frameworks. We obtain

\begin{align*}
\dim \ker df_G(p) &= vd - e + \dim \Omega(G,p) \\
                  &= (a+b)d - ab + (a-d-1)(b-d-1) \\
                  &= \binom{d+1}{2} + \binom{d+2}{2}-a-b
\end{align*}

However, the space of quadric surfaces has dimension $\binom{d+2}{2}-a-b$, and since each quadric surface gives rise to an independent flex, they span all infinitesimal flexes modulo Euclidean motions. Because balanced bipartite graphs have this nice property, we will only consider balanced joined graphs for the rest of the paper.

\section{Quadric Rigidity Matrix}

Specifying the coordinates of a single point $p_i$ forces any quadric surface $Q$ containing $p_i$ to satisfy the linear constraint $p_i^T Qp_i = 0$. There is a similar constraint by adding extraneous edges to vertices within the same bipartite class.

\begin{prop}
Given a complete bipartite framework, if $x_i$ and $x_j$ are vertices in the same bipartite class, then adding the edge $x_ix_j$ imposes the linear constraint $p_i^T Qp_j = p_j^T Qp_i = 0$ on the space of quadric surfaces $Q$ whose quadric flex preserves the length of $x_ix_j$.
\end{prop}

\begin{proof}
We wish to find a quadric surface $Q$ such that its squadric flex satisfies the infinitesimal flex condition $(p_i-p_j)\cdot(p_i'-p_j') = 0$. Then

\begin{align*}
(p_i-p_j)\cdot (p_i'-p_j') &= (p_i-p_j) \cdot (Qp_i-Qp_j) \\
                           &= p_i\cdot Qp_i + p_j\cdot Qp_j - p_i\cdot Qp_j - p_j\cdot Qp_i \\
                           &= - p_i\cdot Qp_j - p_j\cdot Qp_i.
\end{align*}

However, since $Q$ is symmetric, $p_i\cdot Qp_j = p_i^TQp_j = p_j^TQp_i = p_j\cdot Qp_i$, so $p_i^T Qp_j = p_j^T Qp_i = 0$.
\end{proof}

To see that both the constraints from vertices and edges are in fact linear, we look at the polynomial form for a quadric. Suppose we have a configuration that maps vertices $x$ and $y$ to $p$ and $q$ in $\R^d$, respectively. For the vertex constraint of the vertex $x$, we obtain

\[\displaystyle\sum^d_{i=1}A_ip_i^2 + \displaystyle\sum_{j=1}^d\sum_{k=j+1}^d 2B_{j,k}(p_jp_k) + \displaystyle\sum^d_{l=1}2C_lp_l + 1 = 0.\]

Similarly, the edge constraint of the edge $xy$ yields

\[\displaystyle\sum^d_{i=1}A_i(p_iq_i) + \displaystyle\sum_{j=1}^d\sum_{k=j+1}^d B_{j,k}(p_jq_k+p_kq_j) + \displaystyle\sum^d_{l=1}C_l(p_l + q_l) + 1 = 0,\]

where $A_i, B_{j,k}$ and $C_l$ are variables representing the coefficients of the quadric polynomial. We define the \emph{(d-dimensional) constraint mapping} $m : \R^{2d} \to \R^{\binom{d+2}{2}}$ where $(p,q)$ is mapped to

\begin{align*}
(p_1q_1,\hspace{1pt}p_2q_2,\hspace{1pt}...\hspace{1pt},\hspace{1pt}p_dq_d,& \\
p_1q_2+p_2q_1,\hspace{1pt}&p_1q_3+p_3q_1,\hspace{1pt}...\hspace{1pt},\hspace{1pt}p_{d-1}q_d+p_dq_{d-1}, \\
&\hspace{50pt}p_1+q_1,\hspace{1pt}p_2+q_2,\hspace{1pt}...\hspace{1pt},\hspace{1pt}p_d+q_d, 1).
\end{align*}

These are the coefficients of the $A_i, B_i$, and $C_i$ variables in the edge constraint. For a single point, the constraint mapping is $m(p, p)$. Since the vertex and edge constraints form a system of linear equations, it is natural to define the following. 

\begin{defn}
Let $G+H$ be a joined graph with vertex set $V$ and extraneous edge set $E'$. The \emph{quadric rigidity matrix} (QRM) of $G+H$ is the $(|V|+|E'|)\times\binom{d+2}{2}$ matrix whose rows are the constraint mappings $m(v,v)$ for all $v \in V$ and $m(v_i,v_j)$ for all $v_iv_j \in E'$.
\end{defn}

Since the quadric flex automatically preserves non-extraneous edge lengths, those edges do not impose any constraint on the QRM. Suppose we have the joined graph $(K_2 \cup K_1) + E_3$ with configuration $p_1 = (4, -5), p_2 = (2,4), p_3 = (-1, 3), p_4 = (-4, -1), p_5 = (-9,0), p_6 = (5, 7)$ such that the extraneous edge connects $v_1$ and $v_2$. We can write the QRM of this joined framework as

\[\bordermatrix{           & x^2     & y^2     & xy      & x       & y       & 1        \cr
                       v_1 & 16      & 25      & -40     & 8       & -10     & 1        \cr
                       v_2 & 4       & 16      & 16      & 4       & 8       & 1        \cr
                       v_3 & 1       & 9       & -6      & -2      & 6       & 1        \cr
                       v_4 & 16      & 1       & 8       & -8      & -2      & 1        \cr
                       v_5 & 81      & 0       & 0       & -18     & 0       & 1        \cr
                       v_6 & 25      & 49      & 24      & 10      & 14      & 1        \cr
                    v_1v_2 & 8       & -20     & 6       & 6       & -1      & 1        }.\] 

The following results are crucial for the main results of this paper.

\begin{prop}
The quadric rigidity matrix of a joined framework $(G,p)$ has rank $\binom{d+2}{2}$ if and only if $(G,p)$ is infinitesimally rigid.
\end{prop}
\begin{proof}
The rank is less than $\binom{d+2}{2}$ if and only if a quadric surface satisfying those constraints exists. Since balanced complete bipartite frameworks only have quadric flexes by Proposition \ref{inout}, a balanced joined graph is flexible if and only if there is a quadric surface satisfying all constraints.
\end{proof}

\begin{prop}
\label{eq}
Let $G_1=G+H$ and $G_2=G'+H'$ be two balanced joined graphs where $G \cup H$ is isomorphic to $G' \cup H'$. That is, the resulting graphs from deleting all non-extraneous edges are isomorphic. Then $G+H$ is generically locally rigid if and only if $G'+H'$ is generically locally rigid.
\end{prop}
\begin{proof}
$(G_1,p)$ and $(G_2,p)$ have the same QRM, since the matrix is only dependent on the extraneous edges and the vertices, and not on the non-extraneous edges.
\end{proof}

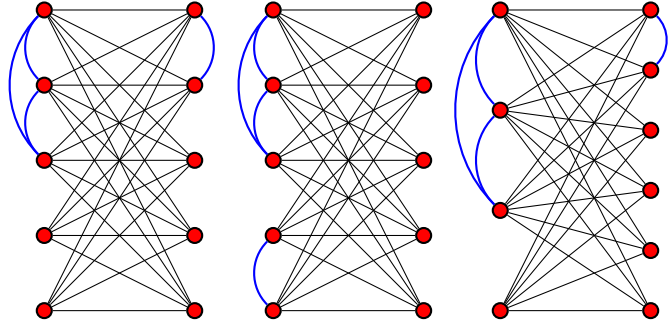
\begin{figure}[ht]

\begin{tikzpicture}
[v/.style={circle,draw=black!100,fill=red!100,thick,inner sep=2pt}]
\foreach \i in {0,...,4}
  \node (v\i) at (0,2-\i)[v]{};
\foreach \i in {5,...,9}
  \node (v\i) at (2,7-\i)[v]{};
\foreach \i in {0,...,4}
  \foreach \j in {5,...,9}
    \draw [-] (v\i) to (v\j);
\draw[blue, thick, bend right=45] (v0) to (v1);
\draw[blue, thick, bend right=45] (v0) to (v2);
\draw[blue, thick, bend right=45] (v1) to (v2);
\draw[blue, thick, bend left=45] (v5) to (v6);
\end{tikzpicture}
\begin{tikzpicture}
[v/.style={circle,draw=black!100,fill=red!100,thick,inner sep=2pt}]
\foreach \i in {0,...,4}
  \node (v\i) at (0,2-\i)[v]{};
\foreach \i in {5,...,9}
  \node (v\i) at (2,7-\i)[v]{};
\draw[blue, thick, bend right=45] (v0) to (v1);
\draw[blue, thick, bend right=45] (v0) to (v2);
\draw[blue, thick, bend right=45] (v1) to (v2);
\draw[blue, thick, bend right=45] (v3) to (v4);
\foreach \i in {0,...,4}
  \foreach \j in {5,...,9}
    \draw [-] (v\i) to (v\j);
\end{tikzpicture}
\begin{tikzpicture}
[v/.style={circle,draw=black!100,fill=red!100,thick,inner sep=2pt}]
\foreach \i in {0,...,3}
  \node (v\i) at (0,2-4/3*\i)[v]{};
\foreach \i in {4,...,9}
  \node (v\i) at (2,5.2-4/5*\i)[v]{};
\draw[blue, thick, bend right=45] (v0) to (v1);
\draw[blue, thick, bend right=45] (v0) to (v2);
\draw[blue, thick, bend right=45] (v1) to (v2);
\draw[blue, thick, bend left=45] (v4) to (v5);
\foreach \i in {0,...,3}
  \foreach \j in {4,...,9}
    \draw [-] (v\i) to (v\j);
\end{tikzpicture}
\caption{Since the quadric rigidity matrix is the same for all three graphs, they are either all rigid, or all flexible in $\R^3$. }
\label{fig-samerigid}
\end{figure}

Figure \ref{fig-samerigid} shows three graphs that satisfy the conditions in Proposition \ref{eq}. In particular, all three graphs are generically locally rigid because ten generic points do not lie on a quadric surface in $\R^3$.

The QRM presents a faster method of deciding local rigidity for balanced joined frameworks since the dimensions of the QRM is strictly smaller than that of the rigidity matrix. Let $e'$ be the number of extraneous edges. Then, the number of rows in the QRM is $|V_G|+|V_H|+e'$, which is less than $|V_G||V_H|+e'$, the number of rows in the rigidity matrix. The number of columns in the QRM is $\binom{d+2}{2} < 2(d+1)d \leq vd$, so the QRM is smaller, overall. If we fix the dimension parameter, the complexity of determining local rigidity for a specific configuration using Gaussian elimination is reduced from $O(v^2e)$ to $O(v+e')$.

Recognizing a balanced joined graph takes exponential time by the naive algorithm of checking all balanced partitions of the vertices. We present an $O(|V|^2)$ algorithm. The complete bipartite graph $K_{a,b}$ has edge complement is the graph $K_a \cup K_b$, which has two connected components. A \emph{connected component} is an equivalence class of the relation ``$u$ is connected to $v$." For a balanced joined graph, its edge complement has at least two connected components, which can be partitioned into two sets of size at least $d+1$. Our algorithm uses dynamic programming and runs as follows for a graph $G$:

\begin{enumerate}
\item
If $|V_G| < 2d+2$, reject. 
\item
Compute the edge-complement of $G$.
\item
Using a depth-first search, find the vertex-sets $V_1, V_2..., V_n$ of the connected components of $\overline{G}$. 
\item
Initialize a string array A indexed from 1 to $|V|$.
\item
For each vertex-set $V_i$, do the following. Set A[$|V_i|$] := $``V_i"$. For each $j$ such that A[$j$] is nonempty, set A[$|V_i|+j$] := A[$j$] + ``$V_i$".
\item
If there is an $i$ such that $d+1 \leq i \leq |V_G|-(d+1)$ and A[$i$] is not an empty string, return A[$i$].
\end{enumerate}

Each step takes $O(|V|^2)$ time, so the overall algorithm runs in $O(|V|^2)$ time. At the end of step 5, A[$i$] is nonempty if and only if there exists a partition of the connected components such that the size of one bipartite class is $i$. Step 6 ensures that $i$ is chosen such that the size of both subsets is at least $d+1$. Note that if we did not require the joined graph to be balanced, we would only need to test connectivity on $\overline{G}$. 

\section{Partial Coning of Connelly's Graphs}

Connelly \cite{cn} provided the first known counterexamples to Hendrickson's conjecture by Theorem \ref{cg}. Let the \emph{coning} of a graph $G$ be the graph $G' = G+K_1$. That is, we add a new vertex and connect it to every other vertex. Connelly and Whiteley \cite{cw} demonstrate that the coning operation preserves all the forms of rigidity.

\begin{figure}[ht]
\begin{tikzpicture}
[v/.style={circle,draw=black!100,fill=red!100,thick,inner sep=2pt}]
\foreach \i in {1,...,3}
  \node (v\i) at (0,-1+\i)[v]{};
\foreach \i in {4,...,6}
  \node (v\i) at (1,-4+\i)[v]{};
\node (gl) at (1.5,0.2){$G$};
\draw [-] (v1) to (v2) to (v6) to (v3) to (v5) to (v4) to (v1);
\end{tikzpicture}
\hspace{20pt}
\begin{tikzpicture}
[v/.style={circle,draw=black!100,fill=red!100,thick,inner sep=2pt},
 c/.style={circle,draw=black!100,fill=black!100,thick,inner sep=2pt}]
\foreach \i in {1,...,3}
  \node (v\i) at (0,-1+\i)[v]{};
\foreach \i in {4,...,6}
  \node (v\i) at (1,-4+\i)[v]{};
\node (cc) at (-2,1.5)[c]{};
\node (gl) at (2,0.2){$G+K_1$};
\draw [-] (v1) to (v2) to (v6) to (v3) to (v5) to (v4) to (v1);
\foreach \i in {1,...,6}
  \draw [-,thick] (cc) to (v\i);
\end{tikzpicture}
\caption{A graph and its coning, respectively. }
\label{fig-coning}
\end{figure}
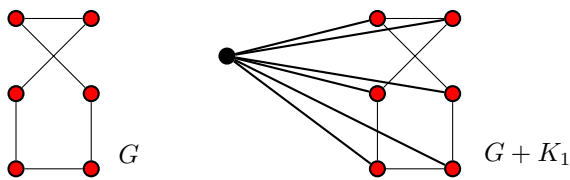

\begin{thm}[Connelly and Whiteley \cite{cw}]
The cone of a graph $G$ is generically [locally, redundantly, globally] rigid in $\R^d$ if and only if $G$ is generically [locally, redundantly, globally] rigid in $\R^{d-1}$.
\end{thm}

\begin{prop}
The cone of a graph $G$ is $(k+1)$-connected if $G$ is $k$-connected.
\end{prop}
\begin{proof}
In the cone of $G$, there are two different ways to delete $k$ vertices. If the cone vertex is deleted, then the result follows immediately from the $k$-connectivity of $G$. If the cone vertex is not deleted, then the resulting graph is still connected because the cone vertex is adjacent to all other vertices.
\end{proof}

In particular, the cone of a graph in Theorem \ref{cg} is also GAGR in the next-highest dimension. It turns out that for those graphs in $\R^5$ and above, only a \emph{partial coning} is necessary. A partial coning is where the cone vertex is joined to only a subset $V' \subsetneq V_G$. We provide a specific type of partial coning that yields a family of GAGR graphs. A partial coning of $K_{9,6}$ as shown in Figure \ref{fig-partial} is the smallest graph in this family.

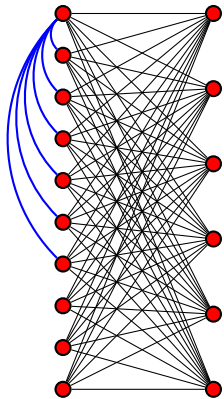
\begin{figure}[ht]
\begin{tikzpicture}
[v/.style={circle,draw=black!100,fill=red!100,thick,inner sep=2pt}]
\foreach \i in {0,...,9}
  \node (v\i) at (0,-1*{5/9*\i})[v]{};
\foreach \i in {0,...,5}
  \node (w\i) at (2,-1*\i)[v]{};
\foreach \i in {0,...,9}
  \foreach \j in {0,...,5}
    \draw [-] (v\i) to (w\j);
\foreach \i in {1,...,6}
\draw[blue, thick, bend right=45] (v0) to (v\i);
\end{tikzpicture}
\caption{A partial coning reimagined as a joined graph. This graph is GAGR in $\R^5$ by Theorem \ref{pc}. }
\label{fig-partial}
\end{figure}

\begin{thm}
\label{pc}
If $a > b\geq d+1$ and $a+b=\binom{d+1}{2}+1$, then the partial coning $(K_{(1,d+1)}\cup E_{(a-d-2)}) + E_b$ is generically almost-globally rigid in $\R^d$ for $d>3$.
\end{thm}

\begin{proof} First, $(K_{(1,d+1)}\cup E_{(a-d-2)}) + E_b$ is in fact a partial coning of a GAGR complete bipartite graph. All extraneous edges are connected to the same vertex. Removal of that vertex leaves $K_{(a-1,b)}$. Since $a > b \geq d+1$, $a-1 \geq (d-1)+2$, so we have a GAGR graph in $\R^{d-1}$.

$(K_{(1,d+1)}\cup E_{(a-d-2)}) + E_b$ is not GGR, since it is a subgraph of a complete cone of a GAGR graph.

The only vertices we need to consider for $(d+1)$-connectivity are the vertices $u_1, u_2,...,u_{a-d+2}$ not connected to the cone vertex $c$. Since $b\geq d+2$, we cannot delete all the vertices of the second bipartite class, leaving at least one vertex $v$ intact. Then the path $u_i-v-c$ connects vertex $u_i$ to the rest of the graph.

When $a=d+2$, we have the coned graph of a GLR graph, which is itself GLR. By Proposition \ref{eq}, $(K_{(1,d+1)}\cup E_{(a-d-2)}) + E_b$ is GLR since it has the same QRM. Adding the $(d+1)$-th edge creates a linear dependency in the QRM because the row-size exceeds $\binom{d+2}{2}$, so that edge is redundant. By symmetry, all the extraneous edges are redundant. By Corollary \ref{bir}, the bipartite edges are redundant because $a,b \geq d+2$. Therefore, the entire graph is GRR.
\end{proof}

The partial conings we considered attached to an entire bipartite class, and furthermore, we only considered partial conings of complete bipartite graphs.

\begin{pro}
Classify all GAGR partial conings of Connelly's graphs, or of other GAGR graphs. 
\end{pro}

Theorem \ref{pc} can be generalized for multiple partial cones of Connelly's graphs. However, this requires a closer manipulation of the QRM that we will encounter in the next section. We conclude this section with an extension that covers weaker partial conings and other classes of GAGR graphs.

\begin{prop}
Let $G$ and $G'$ be generically almost-globally rigid graphs in $\R^d$ such that $G'$ is a factor of $G$. Then any factor $G''$ of $G$ such that $E_{G'} \subseteq E_{G''} \subseteq E_G$ is also generically almost-globally rigid in $\R^d$.
\end{prop}
\begin{proof}
$E_{G'} \subset E_{G''}$ implies $(d+1)$-connectivity and the other conditions follow immediately from Propositions \ref{ggrsub} and \ref{addedge}.
\end{proof}

\section{Graph Attachments in Higher Dimensions}

Frank and Jiang \cite{fj} found a graph that could be ``attached" to other graphs in $\R^5$ to create GAGR graphs. We generalize the result to higher dimensions. For some $x_1,x_2,x_3,x_4\in\N$, let $G_i := E_{x_i}$ for $i \in \{1,2,3,4\}$. The \emph{4-chain} $C_{x_1,x_2,x_3,x_4}$ is the graph with vertex set $V = \bigcup_{i=1}^4V_{G_i}$ and edge set $E = \bigcup_{i=1}^3E_{G_i+G_{i+1}}$. Formally, a 4-chain can be thought of as $G_1 + G_2 + G_3 + G_4$. Frank and Jiang demonstrated that the 4-chain $C_{2,3,5,4}$ could be attached to certain graphs in $\R_5$ to yield GAGR graphs. Attaching a 4-chain to an arbitrary graph $G$, denoted $C_{w,x,y,z} \attach G$, is the result of the vertex amalgamation

\[(C_{w,x,y,z};w_1,w_2,...,w_{x_1}, z_1,z_2,...,z_{x_4})*(G; v_1, v_2,...,v_{x_1+x_4})\]

for some vertices $v_1, v_2,...,v_{x_1+x_4}$ in $V_G$. The vertex amalgamation attaches the vertices of the two ends of the chain to some vertices in $G$. As demonstrated in the proof, the choice of vertices is irrelevant for sufficiently rigid graphs. We can now state the result of Frank and Jiang.

\begin{thm}[Frank and Jiang \cite{fj}]
Let $G$ be a generically redundantly rigid, 6-connected graph. Then $C_{2,3,5,4} \attach G$ is generically almost-globally rigid in $\R^5$.
\end{thm}

\begin{figure}[ht]
\begin{tikzpicture}
[v/.style={circle,draw=black!100,fill=red!100,thick,inner sep=2pt}]
\draw[dashed] (0,0) circle (1);
\draw[dashed] (3,0) circle (1);
\draw[dashed] (0,-3) circle (1);
\draw[dashed] (3,-3) circle (1);
\draw[-,dashed]          (0,-1) to (0,-2);
\draw[-,dashed]          (3,-1) to (3,-2);
\draw[-,dashed]          (1, 0) to (2, 0);
\foreach \i in {0,...,2}
  \node (v\i) at ({cos(\i*120)},{sin(\i*120)})[v]{};
\foreach \i in {0,...,4}
  \node (w\i) at (3+{cos(\i*72)},{sin(\i*72)})[v]{};
\foreach \i in {0,1}
  \node (y\i) at ({cos(90+\i*180)},-3+{sin(90+\i*180)})[v]{};
\foreach \i in {2,...,5}
  \node (y\i) at (3+{cos(\i*90)},-3+{sin(\i*90)})[v]{};
\foreach \i in {0,...,5}
  \foreach \j in {0,...,5}
    \draw [-,blue,thick] (y\i) to (y\j);
\end{tikzpicture}
\caption{The graph attachment $C_{2,3,5,4} \attach K_6$. A circle of vertices represents each independent set of the 4-chain, and a dashed line between independent sets represents a graph join. $K_6$ is highlighted by the bolded edges. }
\label{fig-attach}
\end{figure}
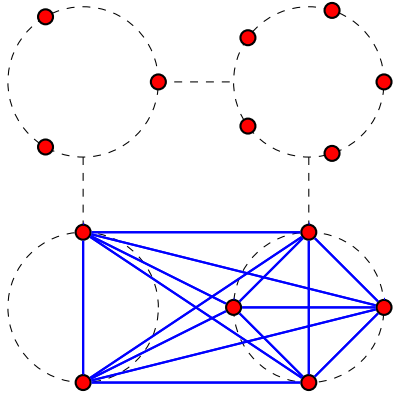

We present the following generalization for higher dimensions.

\begin{thm}
\label{ga}
Let $G$ be a generically redundantly rigid, $(d+1)$-connected graph. Then
\[C_{2,i,2d-2-i,d-1} \attach G,\]
where $2 < i < d-1$, is generically almost-globally rigid in $\R^d$.
\end{thm}
For simplicity, let $\C(i,d) := C_{2,i,2d-2-i,d-1}$. Then $\C(3,5)$ is $C_{2,3,5,4}$. The one part of their proof that does not immediately generalize in higher dimensions involves demonstrating that $C_{2,3,5,4}\attach K_6$ is GLR, in which they provide only a computer-aided proof.

The graph $\C(3,5)\attach K_6$ can be rewritten as the joined graph $(K_2\cup E_5) + (K_4\cup E_3)$. In general, $\C(i,d)\attach K_{d+1}$ is the joined graph $(K_2 \cup E_{(2d-2-i)})+(K_{d-1}\cup E_i)$. Reinterpreting the attachment as a joined graph allows us to apply the QRM.

\begin{lem}
For $d \geq 3$, the joined graph $(K_2 \cup E_{(2d-4)}) + (K_{d-1}\cup E_2)$ is generically locally rigid.
\end{lem}
\begin{proof}
For simplicity, let $H_d = (K_2 \cup E_{(2d-4)}) + (K_{d-1}\cup E_2)$. $H_d$ has $3d-1$ vertices and $1+\binom{d-1}{2}$ extraneous edges, so there are $\binom{d+2}{2}$ rows in the QRM.E

In $\R^2$, $H_2$ is achieved by applying Proposition \ref{lat} on $K_2$, so it is generically locally rigid. Thus, it has no quadric flex, so its QRM is of maximal rank. Although $H_2$ is not a balanced joined graph, all $H_d$ for $d > 2$ are balanced, so showing that the QRM for all $H_d$ has maximal rank is sufficient.

Assume $H_{d-1}$ is generically locally rigid in $\R^{d-1}$. Then consider the graph $H_d$ in $\R^d$. We can achieve this graph from the graph $H_{d-1}$ by adding one vertex in the first bipartite class, adding two vertices in the second bipartite class, and adding $d-2$ extraneous edges to one of the two vertices. In terms of the QRM, we take a $\binom{d+1}{2}\times\binom{d+1}{2}$ matrix and expand to a $\binom{d+2}{2}\times\binom{d+2}{2}$ matrix.

We add $d+1$ columns, namely the $d$ quadratic terms, denoted $Q_1, Q_2,...,Q_d$ (where $Q_i$ corresponds to the product of the $i$-th coordinate and the $d$-th coordinate), and 1 linear term, denoted $L_d$. We also add $d+1$ rows, formed by adding the three new vertices and $d-2$ extraneous edges. We only need to show that there exists some framework whose QRM has maximal rank\footnote{While the proof uses a framework that would have additional infinitesimal flexes (see Whiteley \cite{wh}), we are only interested in showing that the graph has no quadric flexes as the other flexes are not possible in generic frameworks.}.

\begin{figure}[ht]
\begin{tikzpicture}
[v/.style={circle,draw=black!100,fill=red!100,thick,inner sep=2pt},
 n/.style={circle,draw=black!100,fill=black!100,thick,inner sep=2pt}]
\draw[dashed] (0,0) circle (1);
\draw[dashed] (3,0) circle (1);
\draw[dashed] (0,-3) circle (1);
\draw[dashed] (3,-3) circle (1);
\draw[-,dashed]          (0,-1) to (0,-2);
\draw[-,dashed]          (3,-1) to (3,-2);
\draw[-,dashed]          (1, 0) to (2, 0);
\draw[-,dashed]          (1,-3) to (2,-3);
\foreach \i in {0,...,2}
  \node (v\i) at ({cos(\i*120)},{sin(\i*120)})[v]{};
\foreach \i in {0,...,4}
  \node (w\i) at (3+{cos(\i*72)},{sin(\i*72)})[v]{};
\foreach \i in {0,1}
  \node (y\i) at ({cos(90+\i*180)},-3+{sin(90+\i*180)})[v]{};
\foreach \i in {0,...,3}
  \node (z\i) at (3+{cos(\i*90)},-3+{sin(\i*90)})[v]{};
\node (w5) at (3+{cos(24)},{sin(24)})[n]{};
\node (w6) at (3+{cos(48)},{sin(48)})[n]{};
\node (z4) at (3+{cos(45)},-3+{sin(45)})[n]{};
\draw [-,blue,thick] (y0) to (y1);
\foreach \i in {0,...,4}
  \foreach \j in {0,...,4}
    \draw [-,blue,thick] (z\i) to (z\j);
\end{tikzpicture}
\caption{The induction step as applied to $C_{2,3,5,4}$. The new vertices are shown in a different color. }
\label{fig-induction}
\end{figure}
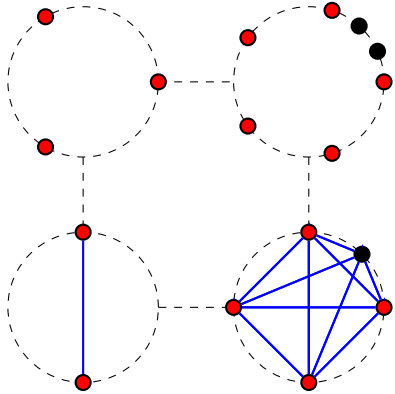

Select a generic framework for $H_{d-1}$ in $\R^{d-1}$. We include that framework into $\R^d$ such that a vertex $v = (v_1, v_2,..., v_{d-1})$ is mapped to $(v_1, v_2,..., v_{d-1}, 0)$. We add three new vertices with coordinates

\begin{align*}
a&=(0, 0, ..., 0, a_{d-1}, a_d) \\
b&=(0, 0, ..., 0, b_{d-1}, b_d) \\
c&=(0, 0, ..., 0, c_{d-1}, c_d)
\end{align*}

and edges $e_1, e_2,...,e_{d-2}$ all connected to $c$. Since the original vertices and extraneous edges have $0$ in the last coordinate, their values in $Q_i$ and $L_d$ are necessarily 0. It suffices to show that the matrix formed by the new rows and columns are of maximal rank, because the matrix is block triangular.

Since we chose all but the last two coordinates to be 0 for vertices $a,b,$ and $c$, their constraint mappings must have 0's in the $Q_2,Q_2,...,Q_{d-1}$ columns. Once again, this is a block triangular matrix, so we need to show that the edges are independent in those $d-2$ columns, and then the vertices in the remaining 3 columns. 

The edges connect vertex $c$ to a vertex in the original graph, so the submatrix formed by the $Q_2,Q_3,...,Q_{d-1}$ columns and the $e_1, e_2,...,e_{d-2}$ rows are coordinates from the original framework all multiplied by $c_d$. Since the determinant is an algebraic equation on the coordinates, it must be non-zero since we selected a generic framework in $\R^{d-1}$. We conclude that the submatrix is of maximal rank as long as $c_d\neq 0$.

We are left with the submatrix formed by the $Q_{1}, Q_d, L_d$ columns and the $a, b, c$ rows:

\[ \left( \begin{array}{ccc}
\vspace{5pt} a_d^2 & 2a_{d-1}a_d & 2a_d \\
\vspace{5pt} b_d^2 & 2b_{d-1}b_d & 2b_d \\
c_d^2 & 2c_{d-1}c_d & 2c_d \\
\end{array} \right).\] 

We may choose any coordinates that makes the submatrix invertible and has $c_d\neq 0$. Since the $(d+1)\times (d+1)$ submatrix is of maximal rank, the entire QRM is of maximal rank, as well.
\end{proof}

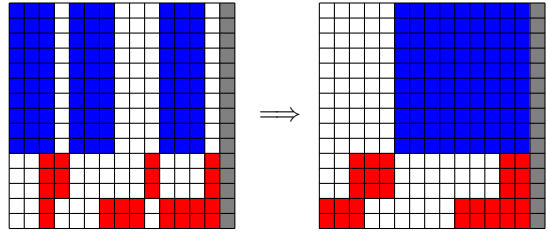
\begin{figure}[ht]
\begin{tikzpicture}
\fill[blue] (0,1) rectangle (3/5,3);
\fill[blue] (4/5,1) rectangle (7/5,3);
\fill[blue] (2,1) rectangle (13/5,3);
\fill[red]  (2/5,0) rectangle (3/5,2/5);
\fill[red]  (6/5,0) rectangle (9/5,2/5);
\fill[red]  (2,0) rectangle (14/5,2/5);
\fill[red]  (2/5,2/5) rectangle (4/5,1);
\fill[red]  (9/5,2/5) rectangle (2,1);
\fill[red]  (13/5,2/5) rectangle (14/5,1);
\fill[black!50]  (14/5,0) rectangle (3,3);
\foreach \i in {0,...,15}{
  \draw [-] (0,\i /5) to (3,\i /5);
  \draw [-] (\i /5,0) to (\i /5,3);
}
\node (su) at (18/5, 1.5){$\Longrightarrow$};
\end{tikzpicture}
\begin{tikzpicture}
\fill[blue] (1,1) rectangle (14/5,3);
\fill[red]  (0,0) rectangle (3/5,2/5);
\fill[red]  (9/5,0) rectangle (14/5,2/5);
\fill[red]  (12/5,2/5) rectangle (14/5,1);
\fill[red]  (2/5,2/5) rectangle (1,1);
\fill[black!50]  (14/5,0) rectangle (3,3);
\foreach \i in {0,...,15}{
  \draw [-] (0,\i /5) to (3,\i /5);
  \draw [-] (\i /5,0) to (\i /5,3);
}
\end{tikzpicture}
\caption{The induction step on $H_3$ to get $H_4$, where non-empty entries are marked. The matrix columns are rearranged to illustrate the block triangular form. }
\label{matrix}
\end{figure}

By applying Proposition \ref{eq}, we obtain generic local rigidity for the graph attachments in consideration. The remainder of the proof is almost identical to the specific case of $\C(3,5)$ in $\R^5$.\footnote{See Frank and Jiang \cite{fj} for a complete proof.}

\begin{proof}
We must show that the attached graph is $(d+1)$-connected, generically redundantly rigid, and not generically globally rigid.

Since $G$ is $(d+1)$-connected, the only possibility for disconnecting the graph is deleting the vertices from $\C(i,d)$. However, we would have to delete all the vertices of $G_1$ and $G_4$, $G_1$ and $G_3$, or $G_2$ and $G_4$, and each of those pairs has at least $d+1$ vertices. Thus, the attached graph is $(d+1)$-connected.

When $G=K_j$, where $j \geq d+1$, the attachment is GLR by repeated application of Proposition \ref{lat}. For the general case, $|V_G| \geq d+1$, so we can compare the flexes of $\C(i,d)\attach G$ to $\C(i,d)\attach K_{|V_G|}$. Suppose a non-trivial flex of $\C(i,d)\attach G$ exists. Since $G$ is assumed to be generically locally rigid, that flex must be a Euclidean motion on $G$. However, this same flex could be applied to $\C(i,d)\attach K_{|V_G|}$ and still be non-trivial, so no such flex exists.

By Proposition \ref{eq}, $\C(i,d) \attach K_{d+1}$ is GLR in $\R^d$ by moving vertices with no extraneous edges to a different bipartite class. The space of stresses for $\C(i,d) \attach K_{d+1}$ has dimension 

\begin{align*}
\Omega(\C(i,d) \attach K_{d+1}, p) &= e-vd+\binom{d+1}{2} \\
&= (2d-i)(d+i-1) + \binom{d-1}{2} + 1 - (3d - 1)d + \binom{d+1}{2} \\
&= (i-2)(d-i-1).
\end{align*}

However, from Theorem \ref{brz}, the space of stresses for the complete bipartite graph $K_{2d-i,d-i-1}$ is also $(i-2)(d-i-1)$. That implies that any non-zero stress on $\C(i,d) \attach K_{d+1}$ is zero on extraneous edges and possibly non-zero on bipartite edges. By Proposition \ref{red}, the extraneous edges are not redundant, while the bipartite edges are. The edge set of $\C(i,d)\attach G$ is the disjoint union of the edge sets of $\C(i,d)$ and $G$. Removing an edge from $G$ leaves the graph locally rigid since $G$ is redundantly rigid. Every edge $e$ of $\C(i,d)$ is redundant in the graph $\C(i,d) \attach K_j$, so $\C(i,d)-\{e\} \attach K_j$ is generically locally rigid. By the same argument, any flex of $\C(i,d)-\{e\} \attach G$ is a flex of $\C(i,d)-\{e\} \attach K_{|V_G|}$, so $\C(i,d)-\{e\} \attach G$ is generically locally rigid. Thus, we have generic redundant rigidity for $\C(i,d)\attach G$.

Since not all edges of $\C(i,d)\attach K_{d+1}$ are redundant, the graph is not generically globally rigid by Theorem \ref{hend}. For any two equivalent frameworks of $G$, two corresponding globally rigid subframeworks are congruent, so adding a vertex to a graph and attaching it to vertices of a globally rigid subgraph preserves global non-rigidity. Thus, $\C(i,d)\attach K_j$ where $j > d+1$ is not GGR. Since $G$ is a factor of $K_{|V_G|}$, $\C(i,d)\attach G$ is not generically globally rigid either by Proposition \ref{ggrsub}.
\end{proof}

We can also exhibit more graph attachments based on different families of 4-chains. In general, for graph attachments that attach to $d+1$ vertices, we found that the sum of the middle two arguments of the 4-chain must be $x(d-x+1)$, which we denote $v(d,x)$. The following generalization can be proven using the same techniques.

\begin{thm}
\label{gax}
Let $G$ be a generically redundantly rigid, $(d+1)$-connected graph. Then
\[C_{x,i,v(d,x)-i,(d+1)-x} \attach G,\]
where $x < i < v(d,x)+x-d-1$, is generically almost-globally rigid in $\R^d$.
\end{thm}

However, there are still 4-chain graph attachments that escape this characterization, namely those that attach to $d+2$ or more vertices (there cannot exist any that attach to only $d$ vertices since this violates $(d+1)$-connectivity). The smallest such outlier we found was $C_{3,3,5,5}$ in $\R^6$. This motivates the following problem.

\begin{pro}
Characterize all 4-chain graph attachments.
\end{pro}

\end{document}